\documentclass[11pt]{amsart}
\usepackage{amsfonts}

\usepackage{amssymb,amsthm,amsmath}
\usepackage{enumerate}

\title[Simplicial complements]{Homology groups of simplicial complements: A new proof of Hochster theorem}
\author[J.~Ma, F.~Fan \& X.~Wang]{Jun Ma, Feifei Fan and Xiangjun Wang}
\address[J.~Ma]{School of Mathematics and Statistics, Fudan Univ., P. R. China}

\email{12110180065@fudan.edu.cn}
\address[F.~Fan]{School of Mathematics and Statistics, Nankai Univ., P. R. China}

\email{fanfeifei@mail.nankai.edu.cn }

\address[X.~Wang]{School of Mathematics and Statistics, Nankai Univ., P. R. China}

\email{xjwang@nankai.edu.cn }

\thanks{The authors are supported by NSFC grant No. 11261062, NSFC grant No. 11371093 and SRFDP No.20120031110025.}

\def\w{\widetilde}

\theoremstyle{plain}
\newtheorem{theorem}{Theorem}[section]
\newtheorem{corollary}[theorem]{Corollary}
\newtheorem{lemma}[theorem]{Lemma}
\newtheorem{proposition}[theorem]{Proposition}

\theoremstyle{definition}
\newtheorem{definition}[theorem]{Definition}

\newtheorem{remark}[theorem]{Remark}

\begin{document}
\newcommand{\FF}{\mathcal{F}}
\newcommand{\cp}{\mathbb{CP}}
\newcommand{\zz}{\mathbb{Z}}
\newcommand{\nn}{\mathbb{N}}
\newcommand{\ext}{{\rm Ext}}
\newcommand{\PP}{\mathbb{P}}
\newcommand{\QQ}{\mathbb{Q}}
\newcommand{\huaa}{\mathscr{A}}
\newcommand{\kk}{\mathbf{k}}
\newcommand{\ba}{\mathbf{a}}
\newcommand{\xx}{\mathbf{x}}
\newcommand{\uu}{\mathbf{u}}
\newcommand{\UU}{\mathbf{U}}
\newcommand{\vv}{\mathbf{v}}
\newcommand{\ZZ}{\mathcal{Z}}
\newcommand{\II}{\mathcal{I}}
\newcommand{\NN}{\mathcal{N}}
\newcommand{\MM}{\mathcal{M}}
\newcommand{\CC}{\mathbb{C}}
%%% ----------------------------------------------------------------------

\begin{abstract}
In this paper, we consider homology groups induced by the exterior algebra generated by a simplicial compliment of a simplicial complex $K$. These homology groups are isomorphic to the Tor-groups $\mathrm{Tor}_{i, J}^{\kk[m]}(\kk(K),\kk)$ of the face ring $\kk(K)$, which is very useful and much studied in toric topology. By using \v Cech homology theory and Alexander duality theorem, we prove that
these homology groups have dualities with the simplicial cohomology groups of the full subcomplexes of $K$.  Then we give a new proof of Hochster's theorem.
\end{abstract}
%%% ----------------------------------------------------------------------
\maketitle

\section{Introduction}
Throughout this paper, $\kk$ is a field or the ring of integers $\mathbb{Z}$. $\kk[m]=\kk[v_1,\dots,v_m]$ is the graded polynomial algebra on $m$ variables,
deg$(v_i)=2$. The \emph{face ring} (also known as the \emph{Stanley-Reisner ring}) of a simplicial complex $K$ with $m$ vertices is the quotient ring
\[\kk(K)=\kk[m]/\mathcal {I}_K\]
where $\mathcal {I}_K$ is the ideal generated by those square free monomials $v_{i_1}\cdots v_{i_s}$
for which $\{i_1,\dots,i_s\}$ is not a simplex in $K$.

For any simple polytope $P^n$, Davis and Januszkiewicz \cite{DJ91} introduced a $T^m$-manifold  $\mathcal{Z}_P$ with orbit space $P^n$. After that Buchstaber and Panov \cite{BP00} generalized this definition to any simplicial complex $K$ with vertex set $[m]=\{1,2,\dots,m\}$, and named it the \emph{moment-angle complex} associated to $K$:
\[\ZZ_K=\bigcup\limits_{\sigma\in K} D(\sigma),\]
where $D(\sigma)= Y_1{\times} Y_2 {\times} \cdots {\times} Y_m $, $Y_i=D^2$ if $i\in\sigma$ and $Y_i=S^1$ if $i\not\in\sigma$.
The following theorem is proved by Buchstaber and Panov \cite{BP00} for the case over a field by using Eilenberg-Moore Spectral Sequence, \cite{BBP04} for the general case.
\begin{theorem}[Buchstaber and Panov {\cite[Theorem 4.7]{P08}}]

Let K be a simplicial complex with $m$ vertices. Then the following isomorphism of algebras holds:
$$H^*(\ZZ_{K};\kk)\cong\mathrm{Tor}^{\kk[m]}(\kk(K),\kk).$$
\end{theorem}

Since $\mathrm{Tor}^{\kk[m]}(\kk(K),\kk)$ has a natural $\zz\oplus\zz^m$-bigrade.
So the bigraded cohomology ring can be decomposed as follows:
$$H^*(\ZZ_{K};\kk)\cong \mathrm{Tor}^{\kk[m]}(\kk(K),\kk)=\bigoplus\limits_{i\geq0}\bigoplus\limits_{J\subseteq[m]}
\mathrm{Tor}_{i,J}^{\kk[m]}(\kk(K),\kk).$$

Hochster \cite{H75} gave a combinatorial description of the Tor-groups $\mathrm{Tor}_{i, J}^{\kk[m]}(\kk(K),\kk)$.
\begin{theorem}[Hochster \cite{H75}]
$$\mathrm{Tor}_{i,\,J}^{\kk[m]}(\kk(K),\kk)\cong\w{H}^{ |J|-i-1}(K_J;\kk),$$
where $\w{H}^{-1}({\varnothing};\kk)=\kk$, $K_J$ is the full subcomplex of $K$ corresponding to $J$.
\end{theorem}
Hochster proved this for the case $\kk$ is a field by analyzing the betti number of both sides. Buchstaber and Panov \cite[Theorem 3.2.9]{BP13} generalize this result to $\kk=\mathbb{Z}$ by means of the Koszul resolution $\Lambda[m]\otimes\kk[m]$ of $\kk$.

Recently, Wang and Zheng \cite{WZ13} gave another way to calculate $\mathrm{Tor}^{\kk[m]}(\kk(K),\kk)$ by using Taylor resolution on the Stanley-Reisner ring $\kk(K)$. This method was presented firstly by Yuzvinsky in \cite{Y99}. In this paper we give a combinatorial description of this method, and then we prove Hochster's theorem in a different way.
\begin{definition}[missing faces and simplicial complements]

Let $K$ be a simplicial complex on the set $[m]$. A \emph{missing face} of $K$ is the subset $\tau \subseteq [m]$ where $\tau \notin K$ and every proper subset of $\tau$ is a simplex of $K$. Denote by $MF(K)$ the missing face set of $K$. Clearly $MF(K)$ is uniquely determined by $K$.

A \emph{simplicial complement} $\mathbb{P}$ of $K$ is a subset of $\{\tau\subseteq [m]\mid\tau\not\in K\}$ so that $MF(K)\subseteq \mathbb{P}$. Usually $\mathbb{P}$ is not unique. Denote by $P(K)$ the set of simplicial complements of $K$.
\end{definition}

Given a simplicial complement $\mathbb{P}$ of $K$, one can define an exterior algebra $\Lambda^*[\mathbb{P}]$ generated by all elements of $\mathbb{P}$.
For a monomial $\uu=\tau_{i_1}\tau_{i_2}\cdots\tau_{i_n}\in \Lambda^*[\mathbb{P}]$,
let \[S_{\uu}=\tau_{i_1}\cup\tau_{i_2}\cup\cdots\cup\tau_{i_n}\] be the \emph{total set} of $\uu$.
So $\Lambda^*[\mathbb{P}]$ has a natural $\zz\oplus\zz^m$-bigrade, which means
$$\Lambda^*[\mathbb{P}]=\bigoplus\limits_{i\geq 0}\bigoplus\limits_{J\subseteq[m]}\Lambda^{i, J}[\mathbb{P}]$$
where $\Lambda^{i, J}[\mathbb{P}]$ is generated by monomials $\uu$  satisfying $S_{\uu}=J$ and deg$(\uu)=i$.

One can make $\Lambda^*[\mathbb{P}]$ into a chain complex by defining a differential $d$ on it. The differential $d: \Lambda^{n,*}[\mathbb{P}]\to \Lambda^{n-1,*}[\mathbb{P}]$ is generated by
$$d(\uu)=\sum_{j=1}^n(-1)^{j+1}\partial_j\uu\cdot\delta_j,$$
where $\uu=\tau_{i_1}\tau_{i_2}\cdots\tau_{i_n}$, $\partial_j\uu=\tau_{i_1}\cdots{\widehat{\tau}_{i_j}}\cdots\tau_{i_n}$ (particularly, if $n=1$, $\partial_1(\tau_{i_1})=1$), $\delta_j=1$ if $S_{\uu}=S_{\partial_i(\uu)}$ and zero otherwise.

\begin{theorem}[see {\cite[Theorem 2.6 and Theorem 3.2]{WZ13}}]\label{thm:1}
Let $K$ be a simplicial complex on the set $[m]$. Given a simplicial complement $\mathbb{P}\in P(K)$. Then
$$\mathrm{Tor}^{\kk[m]}_{i,\,J}(\kk(K),\kk)\cong H_i\big( \Lambda^{*, J}[\mathbb{P}],d\big).$$
\end{theorem}

\begin{remark}
From Theorem \ref{thm:1}, we know that the homology groups
$H_i\big( \Lambda^{*, J}[\mathbb{P}],d\big)$ of a simplicial complement $\mathbb{P}$ is not depend on the choice of $\mathbb{P}$. It just depend on the simplicial complex $K$.
\end{remark}
If we can prove that
$$H_i\big(\Lambda^{*,J}[\mathbb{P}],d\big)\cong\w{H}^{|J|-i-1}(K_J;\kk),$$
then the Hochster theorem is proved. The following theorem ensure these isomorphisms hold.
\begin{theorem}\label{thm:3}
Let $K$ be a simplicial complex on the set $[m]$, and let $\mathbb{P}$ be one of the simplicial compliments of $K$.
Then we have the following group isomorphisms:
$$H_i\big( \Lambda^{*,[m]}[\mathbb{P}],d\big)\cong\w{H}^{m-i-1}(K;\kk).$$

\end{theorem}

\begin{corollary}
Let everything be as before. Then
$$H_i\big( \Lambda^{*, J}[\mathbb{P}],d\big)\cong\w{H}^{|J|-i-1}(K_J;\kk).$$
\end{corollary}

\section{Combinatorial description of homology groups of simplicial complements}

If $K$ is a simplex, the theorem is trivial.
So in this paper,
we assume that $K$ is a simplicial complex on the set $[m]$ but not a simplex.
Given a simplicial complement $\mathbb{P}$ of $K$, suppose
$$\mathbb{P}=\{\tau_1,\tau_2,\ldots,\tau_r\}.$$
For any $\tau_i\in \mathbb{P}$,
we have a simplicial complex \[\mathrm{star}_{\partial\Delta^{m-1}}{\tau_i}:=\{\tau\in\partial\Delta^{m-1}~|~\tau\cup\tau_i\in\partial\Delta^{m-1}\}.\]
Clearly $\mathrm{star}_{\partial\Delta^{m-1}}{\tau_i}$ is a triangulation of $D^{m-2}$.
We denote by $U_i=\mathrm{int}\big|\mathrm{star}_{\partial\Delta^{m-1}}{\tau_i}\big|$,
the interior of the geometric realization of $\mathrm{star}_{\partial\Delta^{m-1}}{\tau_i}$.
\begin{proposition}
Let $U(K):=|\partial\Delta^{m-1}|\setminus|K|$. $\mathbb{P}$ and $U_i$ are defined as above. Then $\UU=\big\{U_i\big\}_{i=1,2,\ldots,r}$
is a open cover of $U(K)$.
\end{proposition}

\begin{proof}
If $x\in|\partial\Delta^{m-1}|\setminus|K|$, there is a simplex $\tau\in 2^{[m]}\setminus K$ satisfying $x$ in the relative interior of $\tau$.
From the definition of a simplicial complement, there exists a simplex $\tau_i\in\mathbb{P}$, such that $\tau_i\in\tau$. Thus, $x\in U_i$.
\end{proof}

\begin{definition}
For any topological space $X$,
Let $\UU=\big\{U_i\big\}_{i\in I}$ be an open cover of the space $X$ indexed by a set $I$.
We define the \emph{nerve} $\NN(\UU)$ to be the abstract simplicial complex on the set $I$
\[
\NN(\UU)=\big\{(i_1, i_2\ldots, i_n)\subseteq I~|~U_{i_1}\cap U_{i_2}\cap\ldots\cap U_{i_n}\neq{\varnothing}\big\}.
\]
For a simplex $\sigma=(i_1,i_2,\dots,i_n)\in \NN(\UU)$, denote by $U_{\sigma}=U_{i_1}\cap U_{i_2}\cap\ldots\cap U_{i_n}$.

The simplicial homology groups of $\NN(\UU)$ are called the \emph{homology groups of the open cover $\UU$}, and denoted by
\[\check{H}_*(X,\UU;\kk):=H_*(\NN(\UU);\kk).\]

\end{definition}

The following theorem is a canonical result of \emph{\v Cech homology} theory which will be used in the sequel.

\begin{theorem}[see {\cite[Corollary 13.3]{Bredon}}]\label{thm:4}
Let $\UU=\big\{U_i\big\}_{i\in I}$ be an open cover of the space $X$ having the property that $\w H_*(U_{\sigma})=0$ for all $\sigma\in \NN(\UU)$. Then there is a canonical isomorphism
\[H_*(X;\kk)\cong \check{H}_*(X,\UU;\kk).\]
\end{theorem}

\begin{theorem}\label{thm:2}
Let $K$ be a simplicial complex on $[m]$,
$\mathbb{P}=\{\tau_1,\tau_2,\ldots,\tau_r\}$ be a simplicial complement of $K$.
By Proposition 2.1,
$\UU=\big\{U_i\big\}_{i=1,2,\ldots,r}$ forms an open cover of the topological space
$U(K)$.
Then  we have the following isomorphisms:
\[
H_n\big( \Lambda^{*,[m]}[\mathbb{P}],d\big)\cong\widetilde{\check{H}}_{n-2}(U(K),\UU;\kk):=\widetilde{H}_{n-2}(\NN(\UU);\kk).
\]
\end{theorem}

Before proving Theorem \ref{thm:2}, we work on the following lemma first.

\begin{lemma}\label{lem:1}
We use notations as above. Then
\begin{enumerate}[(1)]
\item If $\tau_i\cup\tau_j\neq[m]$, then
$U_i\cap U_j=\mathrm{int}\big|\mathrm{star}_{\partial\Delta^{m-1}}\tau_i\cup\tau_j\big|.$

\item If $\tau_i\cup\tau_j=[m]$, then $U_i\cap U_j=\varnothing$.
\end{enumerate}
\end{lemma}

\begin{proof}
Note that for any simplex $\tau$ of $\partial\Delta^{m-1}$,
\[\mathrm{int}\big|\mathrm{star}_{\partial\Delta^{m-1}}\tau\big|=\bigcup\limits_{\tau\subseteq\sigma}\overset{\circ}{|\sigma|},\]
where $\overset{\circ}{|\sigma|}$ is the relative interior of the geometric realization of $\sigma$. Then the lemma follows by a direct checking.
\end{proof}

\begin{proof}[Proof of Theorem \ref{thm:2}]

From Definition 2.2,
we know that
$$\NN(\UU)=\big\{(i_1, i_2\ldots, i_n)\in 2^{[r]}\mid U_{i_1}\cap U_{i_2}\cap\ldots\cap U_{i_n}\neq{\varnothing}\big\}.$$
Lemma \ref{lem:1} shows that
$U_{i_1}\cap U_{i_2}\cap\ldots\cap U_{i_n}\neq{\varnothing}$ if and onely if
$\tau_{i_1}\cup\tau_{i_2}\ldots\cup\tau_{i_n}\neq [m]$. So the nerve can be written as
$$\NN(\UU)=\big\{(i_1, i_2\ldots, i_n)\in 2^{[r]}\mid\tau_{i_1}\cup\tau_{i_2}\ldots\cup\tau_{i_n}\neq [m]\big\}.$$

Let $\Lambda^{*}[\mathbb{P}]$ be the exterior algebra generated by
$\mathbb{P}=\{\tau_1,\tau_2,\ldots,\tau_r\}$. We define another differential $\partial:\Lambda^{n}[\mathbb{P}]\rightarrow \Lambda^{n-1}[P_0]$ generated by
$$\partial(\uu)=\sum_{j=1}^n(-1)^{j+1}\partial_j\uu,$$ where $\partial_j\uu=\tau_{i_1}\cdots{\widehat{\tau}_{i_j}}\cdots\tau_{i_n}$ for any monomial $\uu=\tau_{i_1}\tau_{i_2}\cdots\tau_{i_n}$. Define a homomorphism
$$\Phi:\widetilde{C}_*\big(\NN(\UU);\kk\big)\longrightarrow\Lambda^{*+1}[\mathbb{P}],$$
generated by
$\Phi\big((i_1, i_2,\ldots, i_n)\big):=\tau_{i_1}\tau_{i_2}\cdots\tau_{i_n}\in\Lambda^{n}[\mathbb{P}]$,
where $\widetilde{C}_*\big(\NN(\UU);\kk\big)$ is the reduced simplicial chain complex of $\NN(\UU)$.
Obviously, $\Phi$ is a monomorphism. Let $\Gamma=\mathrm{Im}\Phi$. Then we get a short exact sequence of chain complexes,
$$0\rightarrow\big(\widetilde{C}_*(\NN(\UU);\kk\big),\,\partial)\rightarrow\big(\Lambda^{*+1}[\mathbb{P}],\,\partial\big)
\rightarrow\big(\Lambda^{*+1}[\mathbb{P}]/\Gamma,\,\partial)\rightarrow 0.$$
Apparently, $\Lambda^{*+1}[\mathbb{P}]/\Gamma$ is generated by all monomials $\uu\in \Lambda^{*+1,[m]}[\mathbb{P}]$ (i.e., $S_\uu=[m]$).

It is easy to see that there is a chain isomorphism
$$\big(\Lambda^{*+1}[\mathbb{P}]/\Gamma,\,\partial \big)\cong\big(\Lambda^{*+1,[m]}[\mathbb{P}],\,d\big),$$
where $\big(\Lambda^{*+1,[m]}[\mathbb{P}],\,d\big)$ is as in Theorem \ref{thm:1}.

It is easy to see that $H_*\big(\Lambda^{*}[\mathbb{P}],\,\partial \big)=0$. Thus from the long exact sequence induced by the short exact sequence above, we get that
$$H_n\big( \Lambda^{*,[m]}[\mathbb{P}],\,d\big)\cong\w{H}_{n-2}(\NN(\UU);\kk).$$
\end{proof}

\begin{proof}[Proof of Theorem \ref{thm:3}]
Note first that $\partial\Delta^{m-1}\cong S^{m-2}$, and $\big|\mathrm{star}_{\partial\Delta^{m-1}}\tau\big|\cong D^{m-2}$ for each $\tau\in\partial\Delta^{m-1}$, $\tau\neq\varnothing$.
We combine the results of Theorem \ref{thm:4}, Theorem \ref{thm:2} and Lemma \ref{lem:1} to get that
$$H_n\big( \Lambda^{*,[m]}[\mathbb{P}],\,d\big)\cong\widetilde H_{n-2}(U(K);\kk).$$
From Alexander duality theorem, we have
\[\widetilde H_{n-2}(U(K);\kk)\cong \widetilde H^{m-n-1}(K;\kk).\]
\end{proof}

\end{document}